\documentclass{amsart}

\usepackage{amssymb}
\usepackage{mathrsfs}
\usepackage{xcolor}

\usepackage[colorlinks=true,
            linkcolor=blue,
            citecolor=blue,
            urlcolor=blue]{hyperref}

\usepackage{graphicx}
\usepackage{enumerate}

\theoremstyle{plain}
\newtheorem{theorem}{Theorem}[section]

\newtheorem{proposition}[theorem]{Proposition}
\newtheorem{lemma}[theorem]{Lemma}

\theoremstyle{definition}

\theoremstyle{remark}
\newtheorem{remark}[theorem]{Remark}

\renewcommand{\P}{\mathbf{P}}
\newcommand{\E}{\mathcal{E}}

\newcommand{\N}{\tilde{N}}

\newcommand{\tr}{\operatorname{tr}}
\newcommand{\dv}{\operatorname{div}}
\newcommand{\Ric}{\operatorname{Ric}}

\title[The Penrose inequality with charge for 2-convex initial data sets]{The Penrose inequality with charge for 2-convex initial data sets}

\author{Tuan Dolmen}
\address{Department of Mathematics, Duke University}
\email{tuan.dolmen@duke.edu}

\date{\today}

\begin{document}

\begin{abstract}
We prove the Penrose inequality with charge under the 2-convexity condition recently introduced by Dong. More precisely, given a complete, connected and asymptotically flat Einstein--Maxwell initial data set $(M,g,k; E,B)$ satisfying the charged dominant energy and the 2-convexity conditions, with divergence-free electromagnetic vector fields $(E,B)$ and a connected outermost past apparent horizon $\Sigma$ that satisfies $|\Sigma| \geq 4\pi q^2$ -- where $q$ is the total charge -- we show that the following inequality for the ADM mass $m$ holds:
\[
m\geq \sqrt{\frac{|\Sigma|}{16\pi}} + q^2 \sqrt{\frac{\pi}{|\Sigma|}},
\]
with equality if and only if $k \equiv 0$ and $(M,g;E,B)$ is isometric to a canonical slice of sub-extremal Reissner--Nordstr\"om spacetime. Building on Dong's proof of the uncharged case, we use his $\mathbf{P}$-inverse mean curvature flow and its weak formulation, which only depends on $(g,\mathbf{P})$ and hence applies to the charged setting unchanged. The novelty of our work is the modification of the monotonicity formula to account for the additional charge term. For time-symmetric data ($k \equiv 0$), the flow reduces to the classical inverse mean curvature flow and our monotonicity formula to Jang's monotonicity of the charged Hawking mass, recovering the charged Riemannian Penrose inequality.
\end{abstract}

\maketitle

\section{Introduction}          
The Penrose inequality originates in Penrose's proposed test of the weak cosmic censorship conjecture \cite{Pen69, Pen73}. An isolated body collapses and forms a trapped surface, so that a singularity develops \cite{Pen65}. Assuming cosmic censorship \cite{Pen69}, the singularity is hidden behind an event horizon, whose area is non-decreasing \cite{Haw71}, and the exterior settles down to a stationary black hole. Together with the fact that gravitational radiation carries positive energy, this heuristic bounds the ADM mass of the initial data from below in terms of the area of the initial horizon \cite{Pen73, JW77}. For a general charged collapse, the final state is Kerr--Newman, a rotating charged black hole. The angular momentum can be carried away by gravitational radiation, but charge is conserved. Discarding the angular momentum in the Kerr--Newman area--mass inequality yields the Reissner--Nordstr\"om bound
\begin{equation}\label{eq:Penrose_with_charge}
    m \ge \sqrt{\frac{A}{16\pi}} + q^2 \sqrt{\frac{\pi}{A}}
\end{equation}
where $q$ denotes the total charge, valid on the range $A \ge 4\pi q^2$, on which the right-hand side is increasing.

In light of Dong's new $\P$-inverse mean curvature flow, we prove the Penrose inequality with charge \eqref{eq:Penrose_with_charge} for asymptotically flat initial data sets satisfying the 2-convexity condition of \cite{Don26a}, with equality precisely for canonical slices of sub-extremal Reissner--Nordstr\"om. To our knowledge, this is the first sharp Penrose inequality with charge beyond time symmetry and spherical symmetry \cite{DK12,KMM25}. To state the result precisely, we first fix our conventions.

\subsection*{Conventions and notation.} 
We follow the conventions of \cite{Don26b,Don26a}, augmented by the
electromagnetic quantities introduced below, so that the results there may
be quoted directly. 

Let $(M^3, g, k; E, B)$ be a smooth, connected, asymptotically flat Einstein--Maxwell initial data set, where $g$ is a Riemannian metric, $k$ is a symmetric $(0,2)$-tensor, and $E, B$ are vector fields on $M$. The mass and momentum densities $\mu,J$ are defined by
the constraint equations
\begin{equation}\label{eq:constraint_eqs}
    16\pi\mu = R + (\tr_g k)^2 - |k|_g^2, \qquad
    8\pi J = \dv_g\big(k-(\tr_g k)g\big),
\end{equation}
and we split off the electromagnetic contributions by setting
\begin{equation}\label{eq:matter_densities}
    \mu_{m} := \mu - \tfrac{1}{8\pi}\big(|E|_g^2+|B|_g^2\big), \qquad
    J_{m} := J - \tfrac{1}{4\pi}\,E\times B,
\end{equation}
where $\times$ denotes the cross product of $(M,g)$ with respect to a fixed orientation, $(E\times B)^\flat = \star_g (E^\flat \wedge B^\flat)$. The quantities $\mu_{m}, J_{m}$ are the energy and momentum densities of the
uncharged matter fields. The charged dominant energy condition is 
\begin{equation}\label{eq:charged_DEC}
    \mu_m \geq |J_m|_g
\end{equation}
\begin{remark}\label{rmk:DEC}
The charged dominant energy condition \eqref{eq:charged_DEC} implies the
usual dominant energy condition $\mu \ge |J|_g$: by
$|E\times B|_g \le |E|_g|B|_g$,
\[
\mu - |J|_g \;\ge\; \big(\mu_m - |J_m|_g\big)
+ \tfrac{1}{8\pi}\big(|E|_g - |B|_g\big)^2 \;\ge\; 0 .
\]
\end{remark}
We further require that the electric and magnetic fields satisfy the
divergence-free condition
\begin{equation}\label{eq:E_B_div_free}
    \dv_g E = \dv_g B = 0,
\end{equation}
which expresses the absence of charged matter in the exterior. Given a
closed two-sided surface $S \subset M$, set
\begin{equation}\label{eq:EM_charge_general}
    q_E(S) := \frac{1}{4\pi}\int_S \langle E, \nu\rangle, \qquad
    q_B(S) := \frac{1}{4\pi}\int_S \langle B, \nu\rangle.
\end{equation}
By \eqref{eq:E_B_div_free} and the divergence theorem, $q_E(S)$ and
$q_B(S)$ depend only on the homology class of $S$. Writing
$\Sigma := \partial M$, we define the electric and magnetic charges by
\begin{equation}\label{eq:EM_charge}
    q_E := q_E(\Sigma), \qquad q_B := q_B(\Sigma),
\end{equation}
and the total squared charge by $q^2 := q_E^2 + q_B^2$.

As in \cite{Don26a}, we define the tensor
\begin{equation}\label{eq:P_def}
    \P:= (\tr_g k)g - k
\end{equation}
If $\kappa_3 \geq \kappa_2 \geq \kappa_1$ are the eigenvalues of $k$ with respect to $g$, then $\kappa_3 + \kappa_2 \geq \kappa_3 + \kappa_1 \geq \kappa_2 + \kappa_1$ are the corresponding eigenvalues of $\P$. Hence, the 2-convexity condition 
\begin{equation}
    \kappa_2 + \kappa_1 \geq 0
\end{equation}
is equivalent to $\P\geq 0$. Moreover, for any two-sided surface $S$ with unit normal $\nu$,
\begin{equation}\label{eq:P_normal}
    \P(\nu,\nu) = \tr_g k - k(\nu,\nu) = \tr_S k ,
\end{equation}
so that $\P \ge 0$ is equivalent to $\tr_S k \ge 0$ for every such $S$.

For a closed two-sided surface $S \subset M$ with unit normal $\nu$, the null expansions are given by $\theta_\pm := H_S \pm \tr_S k$, where $H_S$ is the mean curvature of $S$ in $(M,g)$. We call $S$ a \emph{past apparent horizon} if $\theta_- = 0$, i.e. $H_S = \tr_S k$. By \eqref{eq:P_normal}
 this reads
 \begin{equation}\label{eq:horizon}
    H_S = \P(\nu,\nu).
\end{equation}
Under the $2$-convexity condition $\tr_S k \ge 0$, so past apparent horizons
coincide with the generalized apparent horizons $H_S = |\tr_S k|$ of
\cite{BK11}; in particular the $C^{2,\alpha}$
regularity of the outermost such horizon, as well as its outer minimizing
property, are provided by \cite{Eic10} (see also \cite[\S2]{Don26a}).
Reversing the time orientation, $k \mapsto -k$, interchanges past and future
apparent horizons.

Our main result is the following.
\begin{theorem}\label{thm:main}
    Suppose $(M,g,k;E,B)$ is a complete, connected, asymptotically flat
    initial data set satisfying the charged dominant energy condition
    \eqref{eq:charged_DEC}, with $E,B$ satisfying the divergence-free
    condition \eqref{eq:E_B_div_free}. Assume the $2$-convexity condition
    $\P \ge 0$. If $\Sigma = \partial M$ is a connected, outermost past
    apparent horizon with $|\Sigma| \ge 4\pi q^2$, then
    \begin{equation}\label{eq:main}
        m \geq \sqrt{\frac{|\Sigma|}{16\pi}} + q^2 \sqrt{\frac{\pi}{|\Sigma|}},
    \end{equation}
    where $m$ is the ADM mass. Equality holds if and only if $k \equiv 0$ and
    $(M^3,g;E,B)$ is isometric to a canonical slice of sub-extremal
    Reissner--Nordstr\"om.
\end{theorem}

\begin{remark}
The right-hand side of \eqref{eq:main} is increasing in $|\Sigma|$ precisely
on the range $|\Sigma| \ge 4\pi q^2$, where it takes values $\ge |q|$. Below
this range the heuristic above predicts only the mass--charge
inequality $m \ge |q|$, which holds in general \cite{GHHP83}; thus
$|\Sigma| \ge 4\pi q^2$ fixes the regime in which \eqref{eq:main} is expected to hold.
\end{remark}

\begin{remark}
    The divergence-free condition \eqref{eq:E_B_div_free} on the exterior corresponds to all charged matter lying behind the horizon; uncharged matter is permitted, i.e. $\mu_m, J_m$ need not vanish. 
\end{remark}

\begin{remark}
For $q = 0$, Theorem \ref{thm:main} reduces to the connected case of \cite[Theorem 1.2]{Don26a};
for $k \equiv 0$, to the charged Riemannian Penrose inequality
\cite{Jan79,HI01}, with rigidity as in \cite{DK12, KWY17}.
\end{remark}

The time-symmetric inequality was first proved by Jang \cite{Jan79} for a connected horizon under the assumption that a smooth solution of the inverse mean curvature flow exists. Furthermore, Jang observed that \eqref{eq:main} is equivalent to the condition that $m - \sqrt{m^2 - q^2} \leq \rho \leq m + \sqrt{m^2 - q^2}$, where $\rho$ is the area radius, $A = 4\pi \rho^2$, thereby showing that the horizon radius needs to stay between the inner and outer horizon radii of Reissner--Nordstr\"om with parameters $(m,q)$. Only the upper bound follows from the established heuristic. For a disconnected horizon, the lower bound may fail: Weinstein--Yamada constructed a time-symmetric counterexample with $\rho < m - \sqrt{m^2 - q^2}$ \cite{WY05}. The missing existence theory for IMCF in Jang's work was later supplied by the work of Huisken and Ilmanen \cite{HI01}, and the rigidity of the time-symmetric case for a connected horizon was proven by Disconzi--Khuri \cite{DK12}. Under the complementary hypothesis $\rho \geq |q|$, the inequality for multiple charged black holes was later proven by Khuri--Weinstein--Yamada \cite{KWY17} using a charged version of Bray's conformal flow \cite{Bra01}.

Beyond time symmetry, much less is known. Under the additional assumption of spherical symmetry, \eqref{eq:main} was established by Hayward \cite{Hay98}, and by Disconzi--Khuri \cite{DK12} via a generalized Jang equation approach; more recently, Kunduri--Margalef-Bentabol--Muth \cite{KMM25} proved it in all dimensions, with a cosmological constant and its rigidity. Furthermore, Disconzi--Khuri \cite{DK12} reduced the general connected case to a coupled Jang--IMCF system; however, the solvability of this system is still open. Without any symmetry assumption, Khuri \cite{Khu13} proved an unconditional Penrose-like inequality with charge for Einstein--Maxwell data, in which the ADM energy is bounded below by an expression proportional to the square root of the area of the outermost future or past apparent horizon plus the square of the total charge, with constants depending on a solution of a linear elliptic equation. However, this inequality is not sharp.

In particular, no sharp form of \eqref{eq:main} appears to be known off time symmetry beyond spherical symmetry.

When $q=0$, the inequality \eqref{eq:Penrose_with_charge} was recently
established by Dong \cite{Don26a} under the $2$-convexity condition, using
the $\sigma$-inverse mean curvature flow introduced in \cite{Don26b}.

Our main tool is the same flow, namely the $\P$-inverse mean curvature flow of
\cite{Don26a},
\begin{equation}\label{eq:flow}
    \partial_t F \;=\; \frac{\nu}{H - \P(\nu,\nu)},
\end{equation}
together with its weak level-set formulation. The key observation of our proof is that the fields $(E,B)$ affect only the estimates. The weak solutions, outward optimizing hulls, jump structure, and horizon regularity theory associated with \eqref{eq:flow} are defined in terms of $(g,\P)$ alone. Consequently, this entire theory of Dong \cite{Don26b,Don26a} applies unchanged in the charged setting, and the new content of this paper is concentrated in a refined monotonicity formula.

For $N_t$ a smooth, closed, connected solution of the $\P$-IMCF, we define the quantities
\begin{equation}\label{eq:varphi_def}
    \varphi(t):= 1-  \frac{1}{16\pi}\int_{N_t}(H - \P(\nu,\nu))^2
\end{equation}
\begin{equation}\label{eq:q_mass_def}
    m_{\P}^q(N_t):= \sqrt{\frac{|N_t|}{16\pi}} \varphi(t) + q^2 \sqrt{\frac{\pi}{|N_t|}}
\end{equation}
Note that $\varphi = 1$ on the horizon, so $m_{\P}^q(\Sigma)$ is  exactly the Reissner--Nordstr\"om value. 

In \cite{Don26a}, the dominant energy condition enters only through the estimates, where certain terms are bounded below by $\mu - |J|_g$. In the charged setting, however, this discards precisely the electromagnetic contribution. After decomposing the electric and magnetic fields along the flow surfaces, we show that the energy of their normal components survives every estimate. Furthermore, since the enclosed charge is conserved along the flow by \eqref{eq:E_B_div_free}, this energy contributes a definite term: $\varphi(t)$ satisfies
\begin{equation}\label{eq:varphi_deriv}
    \varphi'(t) \;\ge\; -\frac{\varphi(t)}{2} \;+\; \frac{2\pi q^2}{|N_t|}.
\end{equation}

Unlike the classical inverse mean curvature flow (IMCF), for which $|N_t|' = |N_t|$ exactly, the area along the $\P$-IMCF satisfies $|N_t|' = |N_t| + I(t)$, with an extra term $I(t) = \int_{N_t}\P(\nu,\nu)/(H-\P(\nu,\nu)) \ge 0$, which accelerates the decay of the charge term. Differentiating $m^q_{\P}$ and inserting \eqref{eq:varphi_deriv}, the classical cancellations leave
\begin{equation}\label{eq:mono}
    \frac{d}{dt}\,m_{\P}^q(N_t) \;\ge\;
    \frac{I(t)W(t)}{8\sqrt{\pi|N_t|}}, \qquad
    W(t) := \varphi(t) - \frac{4\pi q^2}{|N_t|},
\end{equation} 
so monotonicity of $m^q_{\P}$ is governed by the single scalar function $W$. Note that $W(0) = 1 - 4\pi q^2/|\Sigma|$, so $W(0) \ge 0$ is precisely the assumption $|\Sigma| \ge 4\pi q^2$. Then, by \eqref{eq:varphi_deriv} again, we show that $\big(e^{t/2}W\big)' \ge 0$, so nonnegativity of $W$ propagates from the initial time. The estimates behind \eqref{eq:mono} are integral in form and hold across jump times, so $m^q_{\P}$ is monotone even in the presence of jumps.

Since $|N_t| \ge e^t|\Sigma| \to \infty$, the charge term vanishes in the
limit and $m^q_{\P}(N_t)$ reduces to Dong's $m_{\P}(N_t)$,
whose limit is bounded by $m$ via the asymptotic comparison of
\cite{Don26b}. Combined with the monotonicity and the exact horizon value
$m^q_{\P}(\Sigma)$, this yields \eqref{eq:main}. In the case of
equality with $q \neq 0$, $W$ is strictly positive for
$t > 0$, forcing $I \equiv 0$, $\P \equiv 0$, and $k \equiv 0$; the flow then consists of
round spheres, from which we reconstruct the exterior region of a time-symmetric slice of sub-extremal Reissner--Nordstr\"om directly.

\subsection*{Outline of the paper.}
Section~\ref{sec:prelim} collects the asymptotic conditions, the weak
formulation of the $\P$-inverse mean curvature flow, and the two
electromagnetic estimates on which everything rests,
Lemmas~\ref{lem:EM_pointwise} and~\ref{lem:flux}.
Section~\ref{sec:smooth} carries out the argument for a smooth, closed,
connected solution of the flow, where it is free of analytic difficulties.
Section~\ref{sec:weak} extends the argument to the weak flow:
Section~\ref{sec:weak-growth} proves the weak growth formula,
Theorem~\ref{thm:growth}, an analogue of \cite[Thm.~6.5]{Don26b} in which
the energy term $\int_{N_t} 16\pi\bigl(\mu - J(\nu)\bigr)$ survives, and
Section~\ref{sec:weak-mono} deduces the monotonicity of
$m^{q}_{\P}$, concluding with Theorem~\ref{thm:weak_monotonicity}.
Section~\ref{sec:proof} proves Theorem~\ref{thm:main} and treats the case of
equality.

\subsection*{Acknowledgment.} The author would like to thank Conghan Dong and Hubert Bray for their helpful discussions.

\section{Preliminaries}\label{sec:prelim}
\subsection*{Asymptotics, mass, and charge.} We adopt the asymptotic decay conditions for $(g,k)$ and the resulting definition of the ADM mass $m$ verbatim from \cite{Don26a}. Note that there is no pointwise decay condition imposed on the fields $(E,B)$: the charge $q$ is defined by the flux \eqref{eq:EM_charge} through $\Sigma$ and, by \eqref{eq:E_B_div_free}, agrees with the flux through any surface enclosing $\Sigma$, so the behavior of the fields at infinity plays no role in the argument.

\subsection*{The flow and its weak solutions.}
Since $\P$ is smooth with a uniform $C^{1}$ bound, it is admissible in
the sense of \cite[\S4.1]{Don26b} with $\sigma_{i} \equiv \P$, and no
smoothing is required; the weak level-set formulation of \eqref{eq:flow}
developed in \cite{Don26b}, following \cite{HI01}, therefore applies verbatim.
For a precompact open set $\E_{0} \subset M$ with $C^{2}$ boundary, recall
from \cite[\S4]{Don26b} the notion of a weak solution of \eqref{eq:flow}
with initial condition $\E_{0}$: in particular, $u$ is a locally Lipschitz
function with $\E_{0} = \{u < 0\}$, and we write
\[
  \E_{t} := \{u < t\}, \qquad N_{t} := \partial \E_{t}, \qquad t \geq 0 .
\]
Existence follows from \cite[Thm.~4.9]{Don26b}; the construction, by elliptic
regularization, is recalled in Section~\ref{sec:weak-growth}. Here, and in the
horizon regularity theory of \cite{Eic10}, $M$ is complete without boundary;
when $\Sigma = \partial M$ as in Theorem~\ref{thm:main}, we extend $M$ across
$\Sigma$ and take $\E_{0}$ to be the region $\Sigma$ encloses in the
extension. The choice of extension is immaterial, since the flow remains in
$M$.

We shall use the following properties of the weak flow:
\begin{enumerate}[I)]
  \item $H - \P(\nu,\nu) > 0$ for a.e.\ $t > 0$, $\mathcal H^2$-a.e.\ on $N_t$ \cite[Lem.~4.10]{Don26b}
  \item The area grows at least exponentially:
        $e^{-r}|N_{r}| \leq e^{-s}|N_{s}|$ for all $0 \leq r < s$ \cite[(48)]{Don26b}.
  \item If $\Sigma$ is connected and outermost, then each $N_{t}$ is
        connected, so $\chi(N_{t}) \leq 2$ \cite[Lem.~2.8]{Don26a}.
  \item At a jump time, $N_{t}$ is replaced by the boundary $N_{t}^{+}$ of
        the outward optimizing hull $\E_{t}^{+} = \{u \leq t\}$, and the new portion
        $N_{t}^{+} \setminus N_{t}$ satisfies $H = \P(\nu,\nu)$ \cite[(16)]{Don26a}.
  \item If $\E_{0}$ is bounded by a connected component of the outermost
        past apparent horizon, then
        $\E_{0} = \E_{0}^{+} = \E_{0}'$, where $\E_{0}'$ is
        the strictly minimizing hull of $\E_{0}$ \cite[\S1]{HI01}; the estimates below then extend to $t = 0$ by the smoothing argument of \cite[Rmk.~3.2]{Don26a} and \cite[\S6]{Don26b}.
\end{enumerate}
Since the new boundary $N_{t}^{+}$ in IV) is homologous to $\Sigma$, the
enclosed charge is unchanged across a jump by \eqref{eq:E_B_div_free}.

\subsection*{Electromagnetic estimates.} In order to modify the monotonicity formulae given in \cite{Don26a}, we first split the electric and magnetic fields into their normal and tangent parts along a two-sided surface $S\subset M$ with unit normal $\nu$:
\begin{align*}
    &E = E_\nu \nu + E^T
    \\
    &B = B_\nu\nu + B^T
\end{align*}
Then we can write $\langle E \times B, \nu\rangle_g = \langle E^T \times B^T, \nu\rangle_g$ and hence
\[
\left|\left\langle \frac{1}{4\pi}E\times B, \nu\right\rangle_g\right| \leq \frac{1}{4\pi}|E^T|_g|B^T|_g
\]
We now estimate $\mu - J(\nu)$ which will be critical in our monotonicity computations:

\begin{lemma}\label{lem:EM_pointwise}
Let $S \subset M$ be a two-sided surface with unit normal $\nu$. Then,
pointwise on $S$,
\begin{equation}\label{eq:EM_pointwise}
    \mu - J(\nu) \;\ge\; \frac{1}{8\pi}\big(E_\nu^2 + B_\nu^2\big),
\end{equation}
with equality if and only if $\mu_m = |J_m|_g$ with $J_m(\nu) = |J_m|_g$,
and $|E^T|_g = |B^T|_g$ with
$\langle E^T \times B^T, \nu\rangle_g = |E^T|_g\,|B^T|_g$.
\end{lemma}
\begin{proof}
    Using the decomposition above, we compute
    \begin{align*}
    \mu - J(\nu) &= \bigg(\mu_{m} - J_m(\nu)\bigg) + \frac{1}{8\pi}\bigg[|E|_g^2 + |B|_g^2 - 2\langle E \times B, \nu\rangle_g\bigg]
    \\
    &\geq \bigg(\mu_{m} - J_m(\nu)\bigg) + \frac{1}{8\pi}\bigg[|E|_g^2 + |B|_g^2 - 2|E^T|_g |B^T|_g\bigg]
    \\
    &\geq \bigg(\mu_{m} - |J_m|_g\bigg) + \frac{1}{8\pi}\bigg[|E|_g^2 + |B|_g^2 - 2|E^T|_g |B^T|_g\bigg]
    \\
    &= \underbrace{ \bigg(\mu_{m} - |J_m|_g\bigg)}_{\geq 0 \text{ by charged DEC}} + \frac{1}{8\pi}\bigg[E_\nu^2 + B_\nu^2 + \underbrace{(|E^T|_g - |B^T|_g)^2}_{\geq 0}\bigg]
    \\
    &\geq \frac{1}{8\pi}(E_\nu^2 + B_\nu^2)
\end{align*}
Equality in \eqref{eq:EM_pointwise} forces equality at each of the three
inequalities above. The first gives
$\langle E^T \times B^T, \nu\rangle_g = |E^T|_g |B^T|_g$, the second
$J_m(\nu) = |J_m|_g$, and the third the vanishing of both
$\mu_m - |J_m|_g$ and $(|E^T|_g - |B^T|_g)^2$. Conversely, these conditions
give equality throughout.
\end{proof}

Integrating over a closed surface, the right-hand side of
\eqref{eq:EM_pointwise} is controlled from below by the charge:

\begin{lemma}\label{lem:flux}
Let $S \subset M$ be a closed surface homologous to $\Sigma$. Then
\begin{equation}\label{eq:flux_bound}
    \int_S \big(E_\nu^2 + B_\nu^2\big)
    \;\ge\; \frac{(4\pi q_E)^2 + (4\pi q_B)^2}{|S|}
    \;=\; \frac{16\pi^2 q^2}{|S|},
\end{equation}
with equality if and only if $E_\nu$ and $B_\nu$ are constant on $S$.
\end{lemma}
\begin{proof}
By \eqref{eq:E_B_div_free} and the divergence theorem,
$\int_S E_\nu = 4\pi q_E$ and $\int_S B_\nu = 4\pi q_B$; the claim follows
from the Cauchy--Schwarz inequality
$\big(\int_S E_\nu\big)^2 \le |S| \int_S E_\nu^2$.
\end{proof}

\section{Monotonicity along the smooth flow}\label{sec:smooth}
\subsection{Growth formulae}
For a smooth, closed solution of \eqref{eq:flow},
\begin{equation}\label{eq:area_growth}
\frac{d}{dt}|N_t| = |N_t| + I(t), \qquad
I(t) := \int_{N_t}\frac{\P(\nu,\nu)}{H-\P(\nu,\nu)} \;\ge\; 0,
\end{equation}
so that $e^{-t}|N_t|$ is nondecreasing \cite[\S3]{Don26a}

Using \eqref{eq:varphi_def}, we define\footnote{Note that we redefined $B(t)$ in \cite{Don26a} as $C(t)$ due to notation collision with the magnetic field $B$.}
\begin{equation}
    C(t) := e^{t/2}\varphi(t)
\end{equation}
Then for this quantity we have the following evolution estimate:
\begin{proposition}\label{prop:residue}
Let $(N_t)$ be a smooth, closed, connected solution of \eqref{eq:flow}. Then
\begin{equation}\label{eq:residue}
\begin{split}
16\pi e^{-t/2}C'(t) \;\ge\;
\int_{N_t} 2\Big|\tfrac{\nabla^N(H-\P(\nu,\nu))}{H-\P(\nu,\nu)}
&+ \P^{TN}(\nu,\cdot)\Big|^2 + \big|\mathring{II}+\mathring{\P}^{TN}\big|^2
\\
&\;+\; \int_{N_t}\Big[16\pi\big(\mu - J(\nu)\big)
+ \big(H-\P(\nu,\nu)\big)\,\mathrm{tr}_{N_t}\P\Big]
\end{split}
\end{equation}
where $\P^{TN}(\nu, \cdot)$ is the restriction of $\P(\nu,\cdot)$ to the tangent plane of $N_t$, $\mathring{\P}^{TN} = \P^{TN} - \frac{1}{2}\tr_{N_t}\P^{TN} g_{N_t}$ is the trace-free part, and $\mathring{II} = II - \frac{1}{2}H g_{N_t}$.
\end{proposition}

\begin{proof} 
    We start with the inequality 
    \begin{equation}\label{ineq:Dong}
    \begin{split}
        &16\pi e^{-t/2}C'(t) 
        \\
        &\geq \int_{N_t} 2\left(\frac{\nabla^N (H - \P(\nu,\nu))}{H - \P(\nu,\nu)} + \P^{TN}(\nu,
        \cdot)\right)^2 + |\mathring{II} + \mathring{\P}^{TN}|^2
        \\
        &+ \int_{N_t}\bigg(R - \frac{1}{2}(\P(\nu,\nu))^2 - H \P(\nu,\nu) - 2|\P^{TN}(\nu,\cdot)|^2 + 2(\dv_g \P)(\nu) + H\tr_g \P - |\mathring{\P}^{TN}|^2\bigg)
    \end{split}
    \end{equation}
    as given in \cite[\S3]{Don26a}. We then compute
    \begin{align*}
        &R - \frac{1}{2}(\P(\nu,\nu))^2 - H \P(\nu,\nu) - 2|\P^{TN}(\nu,\cdot)|^2 + 2(\dv_g \P)(\nu) + H\tr_g \P - |\mathring{\P}^{TN}|^2
        \\
        &= R - |\P|_g^2 + \frac{1}{2}(\P(\nu,\nu))^2 - H \P(\nu,\nu) + 2(\dv_g \P)(\nu) + H\tr_g \P + \frac{1}{2}(\tr_{N_t}\P^{TN})^2
        \\
        &= 16\pi(\mu - J(\nu)) + \frac{1}{2}\bigg((\P(\nu,\nu))^2 - (\tr_g \P)^2 + (\tr_{N_t}\P^{TN})^2\bigg) - H \P(\nu,\nu) + H \tr_g \P 
        \\
        &= 16\pi(\mu - J(\nu)) + (\P(\nu,\nu))^2 - \P(\nu,\nu)(\tr_g \P) + H (\tr_g \P - \P(\nu,\nu))
        \\
        &= 16\pi (\mu - J(\nu)) + (H - \P(\nu,\nu))\tr_{N_t}\P
\end{align*}
where the first equality follows from $|\P|^2_g = |\mathring{\P}^{TN}|^2_g + \frac{1}{2}(\tr_{N_t} \P^{TN})^2 + 2|\P^{TN}(\nu,\cdot)|^2_g + |\P(\nu,\nu)|^2$, the second from $16\pi \mu = R - |\P|_g^2 + \frac{1}{2}(\tr_g \P)^2$ and $8\pi J = -\dv_g\P$, the third from cancellations, and the last from $\tr_{N_t}\P = \tr_{g}\P - \P(\nu,\nu)$. Plugging the last expression back into \eqref{ineq:Dong} yields the desired inequality.
\end{proof}
When $\E_0$ is bounded by a connected component of the outermost past
apparent horizon, the monotonicity statements of this section extend to
$t = 0$ by V).

\subsection{Monotonicity of $m_{\P}^q$}
Using Lemmas \ref{lem:EM_pointwise} and \ref{lem:flux}, and Proposition \ref{prop:residue}, we obtain
\begin{align*}
    16\pi e^{-\frac{t}{2}}C'(t) &\geq 
    \int_{N_t} 2\Big|\tfrac{\nabla^N(H-\P(\nu,\nu))}{H-\P(\nu,\nu)}
+ \P^{TN}(\nu,\cdot)\Big|^2 + \big|\mathring{II}+\mathring{\P}^{TN}\big|^2
\\
&\quad \quad \quad \quad \quad   \;+\; \int_{N_t}\Big[16\pi\big(\mu - J(\nu)\big)
+ \big(H-\P(\nu,\nu)\big)\,\mathrm{tr}_{N_t}\P\Big]
\\
&\geq 16\pi \int_{N_t} (\mu - J(\nu))
\\
&\geq 2\int_{N_t} (E_\nu^2 + B_\nu^2)
\\
&\geq \frac{32 \pi^2 q^2}{|N_t|}
\end{align*}
where the second inequality follows from discarding all of the non-negative terms. Note that $H- \P(\nu,\nu) > 0$ and $\tr_{N_t}\P \geq 0$. By definition of $C(t)$, this yields
\begin{equation}\label{ineq:varphi_dot}
    \varphi'(t) \geq -\frac{\varphi(t)}{2} + \frac{2\pi q^2}{|N_t|}
\end{equation}
We now define the function
\begin{equation}
    W(t) := \varphi(t) - \frac{4\pi q^2}{|N_t|}
\end{equation}
Recall the definition from \eqref{eq:q_mass_def}. Then we compute the following:
\begin{align*}
    \frac{d}{dt}m_{\P}^q(N_t) &= \frac{d}{dt}\bigg(\sqrt{\frac{|N_t|}{16\pi}}\varphi(t) + q^2 \sqrt{\frac{\pi}{|N_t|}}\bigg)
    \\
    &= \frac{\varphi(t)}{2}\sqrt{\frac{|N_t|}{16\pi}} + \frac{\varphi(t) I(t)}{2\sqrt{16\pi |N_t|}} + \sqrt{\frac{|N_t|}{16\pi}} \varphi'(t) - \frac{q^2}{2}\sqrt{\frac{\pi}{|N_t|}} - \frac{q^2 \sqrt{\pi} I(t)}{2|N_t|^{3/2}}
    \\
    &\geq \frac{\varphi(t)}{2}\sqrt{\frac{|N_t|}{16\pi}} + \frac{\varphi(t) I(t)}{2\sqrt{16\pi |N_t|}} - \frac{\varphi(t)}{2}\sqrt{\frac{|N_t|}{16\pi}} + \frac{2\pi q^2}{|N_t|}\sqrt{\frac{|N_t|}{16\pi}} - \frac{q^2}{2}\sqrt{\frac{\pi}{|N_t|}} - \frac{q^2 \sqrt{\pi} I(t)}{2|N_t|^{3/2}}
    \\
    &= \frac{I(t)}{8\sqrt{\pi |N_t|}}\left[\varphi(t) - \frac{4\pi q^2}{|N_t|}\right]
    \\
    &= \frac{I(t)W(t)}{8\sqrt{\pi |N_t|}}
\end{align*}
where the third line follows from \eqref{ineq:varphi_dot}. We have shown
\begin{equation}
    \frac{d}{dt}m_{\P}^q(N_t) \geq  \frac{I(t)W(t)}{8\sqrt{\pi |N_t|}}
\end{equation}
Notice that $I(t) \geq 0$, hence it suffices to show that $W(t)\geq 0$ for all $t\geq 0$ to conclude that $m_{\P}^q$ is monotone non-decreasing:
\begin{align*}
    \frac{d}{dt}\bigg(e^{\frac{t}{2}}W(t)\bigg) &= e^{\frac{t}{2}}\bigg[\frac{1}{2}W(t) +  \frac{d}{dt}\bigg(\varphi(t) - \frac{4\pi q^2}{|N_t|}\bigg)\bigg]
    \\
    &= e^{\frac{t}{2}}\bigg[\frac{1}{2} \varphi(t) - \frac{2\pi q^2}{|N_t|} + \varphi'(t)  + \frac{4\pi q^2 I(t)}{|N_t|^2}+ \frac{4\pi q^2}{|N_t|}\bigg]
    \\
    &\geq \frac{e^{\frac{t}{2}} 4\pi q^2}{|N_t|}\bigg(1 + \frac{I(t)}{|N_t|}\bigg) \geq 0
\end{align*}
where the third line follows from \eqref{ineq:varphi_dot}. Also $\varphi(0) =1$, which implies $W(0) = 1 - \frac{4\pi q^2}{|\Sigma|}$. This quantity is non-negative by our assumption in Theorem \ref{thm:main}, thus we get $W(t) \geq 0$ for all $t \geq 0$, which in turn implies
\begin{equation}
    \frac{d}{dt}m_{\P}^q(N_t) \geq 0
\end{equation}

\begin{remark}\label{rmk:strict_function}
If $q \neq 0$, the above shows $\big(e^{t/2}W\big)' > 0$, so that
$W(t) > 0$ for every $t > 0$, even when $W(0) = 0$.
\end{remark}

\section{Monotonicity along the weak flow}\label{sec:weak}
\subsection{Growth formulae}\label{sec:weak-growth}

We prove an analogue of \cite[Thm.~6.5]{Don26b} in which the term
$16\pi\int_{N_t}(\mu - J(\nu))$ is retained rather than discarded; this is the
input to the monotonicity of $m_{\P}^q$ established in
Section~\ref{sec:weak-mono}. The proof is a slight modification of the
argument given in \cite[\S6]{Don26b}. We adopt Dong's notation, adapted to
our setting ($\sigma = \P$).
Let $\tilde{g} = g + dz^2$ be the product metric on $M\times \mathbb{R}$.
Extend $\P$ constantly in the $z$-direction, so that
$\P(e_z,\cdot) = 0$. Symbols with a tilde will denote quantities
computed in $(M\times \mathbb{R}, \tilde{g})$ for the rest of the section. As
noted in Section~\ref{sec:prelim}, $\P$ is admissible in the sense of
\cite[\S4.1]{Don26b} with $\sigma_i \equiv \P$; no smoothing of
$\sigma$ is required, in contrast with the scalar case $\sigma = |h|g$ of
\cite{Don26b}. Let $u^{\epsilon_i}$ denote the solutions of the elliptic
regularization $(\ast)_{\epsilon_i}$ of \cite[\S4.1]{Don26b} on the
precompact domains $\Omega_{L_i} \to M$, where the subsequences
$\epsilon_i \to 0$ and $L_i \to \infty$ are supplied by the proof of
\cite[Thm.~4.9]{Don26b}. We also define
$U_i(x,z) := u^{\epsilon_i}(x) - \epsilon_i z$ and the level sets
$\tilde N_t^i := \{U_i = t\}$, which are smooth solutions of the
$\P$-IMCF in $\Omega_{L_i}\times\mathbb{R}$ for
$-\infty < t < \infty$. Moreover, with
$\tilde\nu_i := \tilde\nabla U_i/|\tilde\nabla U_i|$,
\[
    U_i \to U(x,z) = u(x) \ \text{ locally uniformly}, \qquad
    \tilde\nu_i \to \tilde\nu \ \text{ a.e.}, \qquad
    \tilde N^i_t \to \tilde N_t = N_t\times\mathbb{R}
\]
locally in $C^{1,\alpha}$ for a.e.\ $t \ge 0$, where $N_t = \partial\E_t$
is the weak solution of Section~\ref{sec:prelim}.
All tangential projections along $\tilde N^i_t$ are taken with respect to
$\tilde\nu_i$. We also fix, as in \cite[\S6]{Don26b}, a cutoff function
$\phi\in C^2_c(\mathbb{R})$ with $\phi \ge 0$,
$\operatorname{supp}\phi\subset[2,4]$ and $\int\phi\,dz = 1$, together with a
compact set $K(T)$ containing
$\tilde N^i_t\cap(M\times\operatorname{supp}\phi)$ for all $0\le t\le T$ and
all large $i$. At every non-jump time $\tilde\nu$ is orthogonal to $e_z$, so
that, writing $\beta_i := \langle\tilde\nu_i,e_z\rangle$, we have
$\sup_{\tilde N^i_t\cap(M\times\operatorname{supp}\phi)}|\beta_i| \to 0$. We
shall use the uniform estimates \cite[(32),\,(35)]{Don26b} without further
comment.

As in \cite[(30)]{Don26b}, for the smooth $\P$-IMCF $\tilde N^i_t$,
we get
\begin{align*}
    &\frac{d}{dt}\int_{\N_t^i}
      \phi(\tilde{H}_i - \P(\tilde{\nu}_i,\tilde{\nu}_i))^2
    \\
    &= \int_{\N_t^i}(\tilde{H}_i
      - \P(\tilde{\nu}_i, \tilde{\nu}_i))
      \tilde{\nabla}_{\tilde{\nu}_i}\phi
      - 2\left\langle\tilde{\nabla}^N\phi, \,
      \frac{\tilde{\nabla}^N(\tilde{H}_i
      -\P(\tilde{\nu}_i, \tilde{\nu}_i))}{\tilde{H}_i
      - \P(\tilde{\nu}_i, \tilde{\nu}_i)}\right\rangle
    \\
    &- 2\int_{\N_t^i}\phi\left(
      \frac{|\tilde{\nabla}^N(\tilde{H}_i
      - \P(\tilde{\nu}_i, \tilde{\nu}_i))|^2}{(\tilde{H}_i
      - \P(\tilde{\nu}_i, \tilde{\nu}_i))^2}
      + \bigg(\widetilde{\Ric}(\tilde{\nu}_i, \tilde{\nu}_i)
      + \underbrace{|\tilde{II}_i|^2}_{(\dagger)}\bigg)\right)
    \\
    &-\underbrace{2 \int_{\N_t^i}\phi
      (\tilde{\nabla}_{\tilde{\nu}_i}\P)(\tilde{\nu}_i,
      \tilde{\nu}_i)}_{(\star)}
      - 4\int_{\N_t^i}\phi
      \frac{\P\bigg(\tilde{\nu}_i, \,
      \tilde{\nabla}^N (\tilde{H}_i
      - \P(\tilde{\nu}_i, \tilde{\nu}_i))\bigg)}{\tilde{H}_i
      - \P(\tilde{\nu}_i, \tilde{\nu}_i)}
      + \int_{\N_t^i} \phi \tilde{H}_i (\tilde{H}_i
      - \P(\tilde{\nu}_i, \tilde{\nu}_i)),
\end{align*}
where $\widetilde{\Ric}$ and $\tilde{II}_i$ denote the Ricci curvature
of $\tilde g$ and the second fundamental form of $\tilde N^i_t$ in
$(M\times\mathbb{R},\tilde g)$.

In our proof, the main modification is the treatment of the $(\star)$ term.
In \cite[\S6]{Don26b} this term is estimated in absolute value, which
suffices there since the resulting energy term is discarded using 
$\mu \ge |J|_g$. We instead compute its limit exactly. For a
two-sided hypersurface with unit normal $\nu$, in any dimension, we have the
decomposition
\begin{equation}\label{eq:P_identity}
    (\nabla_{\nu} \P)(\nu, \nu) = (\dv_g \P)(\nu)
    - \dv_{N}(\P^{TN}(\nu, \cdot))
    - H\,\P(\nu, \nu)
    + \langle \P^{TN}, II\rangle .
\end{equation}
We now treat each term in turn.

\subsection*{Term 1.} $(\dv_{\tilde g} \P)(\tilde{\nu}_i)$. Since
$\P$ is smooth with $|\nabla \P| \leq C$ on $K(T)$ and
$\tilde{\nu}_i \to \tilde{\nu}$ a.e., by bounded convergence,
\begin{equation}\label{weak_part1}
    \int_{\N_t^i} \phi (\dv_{\tilde g} \P)(\tilde{\nu}_i)
    \longrightarrow \int_{\N_t}\phi (\dv_g \P)(\tilde{\nu}) .
\end{equation}
Note that this term corresponds to the $J(\nu)$ term we had in the strong
formulation.

\subsection*{Term 2.} $\dv_{\N}(\P^{T\N}
(\tilde{\nu}_i, \cdot))$. Here, we integrate by parts, using
$\tilde{\nabla}^N\phi = \phi'(e_z - \beta_i\tilde{\nu}_i)$ and
$\P(\cdot,e_z) = 0$:
\begin{align*}
    \int_{\N^i_t}\phi\,
    \dv_{\N}(\P^{T\N}(\tilde{\nu}_i, \cdot))
    &= -\int_{\N_t^i}\left\langle\tilde{\nabla}^N\phi,\,
      \P^{T\N}(\tilde{\nu}_i,\cdot)\right\rangle
    \\
    &= \int_{\N_t^i}\phi'\,\beta_i\,
      \P(\tilde{\nu}_i, \tilde{\nu}_i) \to 0,
\end{align*}
where $\beta_i = \langle \tilde{\nu}_i, e_z\rangle$ and since
$\beta_i \to 0$ uniformly on $K(T)$, while $\phi'$, $\P$ and the
areas $|\N^i_t\cap(M\times\operatorname{supp}\phi)|$ are uniformly
bounded.

\subsection*{Term 3.} $\tilde{H}_i\,\P(\tilde{\nu}_i, \tilde{\nu}_i)$.
As justified in \cite[\S6]{Don26b}, $\tilde{H}_i \tilde{\nu}_i
\stackrel{*}{\rightharpoonup} \tilde{H}\tilde{\nu}$; testing against the
fixed bounded vector field
$\phi\,\P(\tilde{\nu},\tilde{\nu})\,\tilde{\nu}$ and using
$\langle\tilde{\nu}_i,\tilde{\nu}\rangle \to 1$ uniformly, this in turn
implies
\begin{equation}
    \int_{\N_t^i}\phi\,\tilde{H}_i\,
    \P(\tilde{\nu}_i, \tilde{\nu}_i)
    \longrightarrow \int_{\N_t}\phi\,\tilde{H}\,
    \P(\tilde{\nu}, \tilde{\nu}) .
\end{equation}

\subsection*{Term 4.} $\langle \P^{T\N_i},
\tilde{II}_i\rangle$. We combine this term with $(\dagger)$,
\[
|\tilde{II}_i|^2
+ \bigl\langle \P^{T\N_i}, \tilde{II}_i\bigr\rangle
= \Bigl|\tilde{II}_i + \tfrac12\,\P^{T\N_i}\Bigr|^2
- \tfrac14\,\bigl|\P^{T\N_i}\bigr|^2
\quad\text{pointwise on }\tilde N^i_t.
\]
Since $|\P^{T\N_i}|^2$ is bounded and continuous in the
point and the normal, for a.e.\ $t$,
\[
\int_{\N_t^i}\phi\,|\P^{T\N_i}|^2
\longrightarrow
\int_{\N_t}\phi\,|\P^{T\N}|^2
\]
and hence
\[
\int_r^s\!\!\int_{\N^i_t}\phi\,|\P^{T\N_i}|^2\,dt
\longrightarrow
\int_r^s\!\!\int_{\N_t}\phi\,|\P^{T\N}|^2\,dt
\]
by bounded convergence. Recall from \cite[\S5]{HI01} and
\cite[(35)]{Don26b} that for a.e.\ $t$ the limit surface carries a weak
second fundamental form $\tilde{II}\in L^2(\tilde N_t)$. The completed
square carries a favorable sign and is lower semicontinuous for a.e.\
$t$ [cf.\ \cite[(5.17)--(5.18)]{HI01}]; hence, by Fatou's lemma,
\[
\int_r^s\!\!\int_{\tilde N_t}\phi\,
\Bigl|\tilde{II}+\tfrac12\,\mathbf P^{T\tilde N}\Bigr|^2\,dt
\;\le\;
\liminf_{i\to\infty}\int_r^s\!\!\int_{\tilde N^i_t}\phi\,
\Bigl|\tilde{II}_i+\tfrac12\,\mathbf P^{T\tilde N_i}\Bigr|^2\,dt .
\]

\medskip

On the cylinder $\N_t = N_t\times\mathbb{R}$ all geometric
quantities are invariant in the $z$-direction; moreover
$\tilde{II}(e_z,\cdot)=0$ and, since $\mathbf P(e_z,\cdot)=0$ and
$\tilde\nu\perp e_z$, the tensor $\mathbf P^{T\tilde N}$ has vanishing
$e_z$-components, so that pointwise
$|\tilde{II}+\tfrac12\mathbf P^{T\tilde N}|^2
=|II+\tfrac12\mathbf P^{TN}|^2$ and
$|\mathbf P^{T\tilde N}|^2=|\mathbf P^{TN}|^2$. The normalization
$\int\phi\,dz = 1$ therefore reduces every double integral above to an
integral over $N_t$. Following \cite{Don26b} verbatim for the remaining
terms, we obtain, for a.e.\ $r \ge 0$ and all $s > r$,
\begin{align*}
    \int_{N_s}(H - \P(\nu,\nu))^2
    &- \int_{N_r}(H - \P(\nu,\nu))^2
    \\
    &\leq -2\int_r^s \int_{N_t}\left(
      \frac{|\nabla^N (H - \P(\nu,\nu))|^2}{(H
      - \P(\nu,\nu))^2}
      + \mathrm{Ric}(\nu,\nu)\right)
    \\
    &\quad
    - 2\int_r^s \int_{N_t}
      \Bigl|II+\tfrac12\,\mathbf P^{TN}\Bigr|^2
    + \frac12\int_r^s \int_{N_t}\bigl|\mathbf P^{TN}\bigr|^2
    \\
    &\quad
    - 2\int_r^s \int_{N_t}\Bigl((\dv_g \P)(\nu)
      - H\,\P(\nu,\nu)\Bigr)
    \\
    &\quad
    - 4\int_r^s\int_{N_t}\frac{\P\big(\nu,\,\nabla^N(H
      - \P(\nu,\nu))\big)}{H - \P(\nu,\nu)}
      + \int_r^s \int_{N_t} H (H - \P(\nu,\nu)).
\end{align*}
Since $II,\mathbf P^{TN}\in L^2(N_t)$ for a.e.\ $t$, we may expand the
square pointwise a.e.,
\[
-2\Bigl|II+\tfrac12\,\mathbf P^{TN}\Bigr|^2
+\tfrac12\bigl|\mathbf P^{TN}\bigr|^2
=-2|II|^2-2\bigl\langle \P^{TN},II\bigr\rangle ,
\]
restoring the grouping
\[
-2\Bigl(\underbrace{\mathrm{Ric}(\nu,\nu) + |II|^2}_{(**)}\Bigr)
-2\Bigl((\dv_g \P)(\nu) - H\,\P(\nu,\nu)
+\underbrace{\langle \P^{TN}, II\rangle}_{(***)}\Bigr).
\]
Since $\dim N_t = 2$, we may now split $(***)$ as
$\langle\P^{TN},II\rangle
= \langle\mathring{\P}^{TN},\mathring{II}\rangle
+ \frac{1}{2}(\tr_{N_t}\P)H$, so that
\[
    -H\,\P(\nu,\nu) + \langle\P^{TN},II\rangle
    = \langle\mathring{\P}^{TN},\mathring{II}\rangle
    + \frac{1}{2}H\,\tr_g\P
    - \frac{3}{2}H\,\P(\nu,\nu),
\]
recovering the identity used in \cite[\S3]{Don26a}. Applying the Gauss
equation to $(**)$,
\begin{align*}
    2(\Ric(\nu,\nu) + |II|^2) &= R + |II|^2 + H^2 - 2K
    \\
    &= R + |\mathring{II}|^2 + \frac{3}{2}H^2 - 2K,
\end{align*}
where $K = K_{12} + \lambda_1\lambda_2$ is the weak sectional curvature of
$N_t$ as in \cite[Lem.~6.1]{Don26b}, we can write
\begin{align*}
    2(\Ric(\nu,\nu) + |II|^2)
    + 2\langle \mathring{\P}^{TN}, \mathring{II}\rangle
    &= R + |\mathring{II}|^2 + \frac{3}{2}H^2 - 2K
      + 2\langle \mathring{\P}^{TN}, \mathring{II}\rangle
    \\
    &= R + \frac{3}{2}H^2 - 2K
      + |\mathring{\P}^{TN} + \mathring{II}|^2
      - |\mathring{\P}^{TN}|^2 ,
\end{align*}
and similarly for the gradient terms. Along with the rest of the arguments
in the proof of \cite[Thm.~6.5]{Don26b} and the algebra of Proposition
\ref{prop:residue}, we have proved an analogue of that theorem with a
stronger lower bound.

\begin{theorem}\label{thm:growth}
Let $(M^3,g,k)$ be a complete, connected, asymptotically flat initial data
set without boundary with $\P \ge 0$, and let $\E_0$ be a
precompact open set with $C^2$-boundary. Then there exists a weak solution
of the $\P$-IMCF with initial condition $\E_0$ such that for
a.e.\ $r \ge 0$ and all $s > r$,
\begin{equation}\label{eq:growth}
\begin{split}
\int_{N_s}\big(H - \P(\nu,\nu)\big)^2
&- \int_{N_r}\big(H - \P(\nu,\nu)\big)^2 \\
\le\; &-\frac{1}{2}\int_r^s\!\!\int_{N_t}\big(H - \P(\nu,\nu)\big)^2
 \;+\; 4\pi\int_r^s \chi(N_t) \\
&- \int_r^s\!\!\int_{N_t} 16\pi\big(\mu - J(\nu)\big)
 \;-\; \int_r^s\!\!\int_{N_t}\big(H - \P(\nu,\nu)\big)\,
 \tr_{N_t}\P \\
&- \int_r^s\!\!\int_{N_t}\big|\mathring{II} + \mathring{\P}^{TN}\big|^2
 \;-\; 2\int_r^s\!\!\int_{N_t}\left|
 \frac{\nabla^N\big(H - \P(\nu,\nu)\big)}{H - \P(\nu,\nu)}
 + \P^{TN}(\nu,\cdot)\right|^2 .
\end{split}
\end{equation}
\end{theorem}

When $\E_0$ is bounded by a connected component of the outermost past
apparent horizon, \eqref{eq:growth} extends to $r = 0$ by V).

\subsection{Monotonicity of $m^q_{\P}$}\label{sec:weak-mono}

Throughout this section $(N_t)_{t \ge 0}$ denotes the weak solution provided
by Theorem \ref{thm:growth}, with initial condition $\E_0$ bounded by
$\Sigma$, and we assume the hypotheses of Theorem \ref{thm:main}. It is
convenient to introduce
\begin{equation}\label{eq:eta_xi_def}
    \eta(t) := e^{-t/2}|N_t|^{1/2}, \qquad
    \xi(t) := \frac{e^{t/2}}{|N_t|}, \qquad
    V(t) := e^{t/2}W(t) = C(t) - 4\pi q^2 \xi(t),
\end{equation}
so that $|N_t| = e^{t}\eta(t)^2$, $\xi = e^{-t/2}\eta^{-2}$ and
$|N_t|^{-1/2} = \eta\, \xi$. In this notation the mass \eqref{eq:q_mass_def}
reads
\begin{equation}\label{eq:mass_eta_form}
    m_{\P}^q(N_t)
    = \frac{\eta(t)\,C(t)}{4\sqrt{\pi}}
    + q^2\sqrt{\pi}\,\eta(t)\,\xi(t).
\end{equation}

For a.e.\ $t \ge 0$ set
\begin{equation}\label{eq:deficit}
\begin{split}
    D(t) \;:=\; &\int_{N_t}\Big|\mathring{II} + \mathring{\P}^{TN}\Big|^2
    \;+\; 2\int_{N_t}\left|
      \frac{\nabla^N\big(H - \P(\nu,\nu)\big)}{H - \P(\nu,\nu)}
      + \P^{TN}(\nu,\cdot)\right|^2 \\
    &+\; 4\pi\big(2 - \chi(N_t)\big)
    \;+\; \int_{N_t}\big(H - \P(\nu,\nu)\big)\,\tr_{N_t}\P \\
    &+\; \left[\,16\pi\int_{N_t}\big(\mu - J(\nu)\big)
      \;-\; \frac{32\pi^2 q^2}{|N_t|}\,\right] .
\end{split}
\end{equation}
Each of the five terms of \eqref{eq:deficit} is nonnegative: the first two
are squares; the third because $N_t$ remains connected by the Connectedness
Lemma \cite[Lem.~2.8]{Don26a}, so that $\chi(N_t) \le 2$; the fourth because
$H - \P(\nu,\nu) \ge 0$ a.e.\ by \cite[Lem.~4.10]{Don26b} and
$\tr_{N_t}\P \ge 0$ by the $2$-convexity condition; and the last
because Lemmas \ref{lem:EM_pointwise} and \ref{lem:flux} give
\begin{equation}\label{eq:charge_chain}
    16\pi\int_{N_t}\big(\mu - J(\nu)\big)
    \;\ge\; 2\int_{N_t}\big(E_\nu^2 + B_\nu^2\big)
    \;\ge\; \frac{32\pi^2 q^2}{|N_t|},
\end{equation}
Here, since $N_t$ is homologous to
$\Sigma$, Lemma \ref{lem:flux} is applicable. In particular $D \ge 0$.

\begin{lemma}\label{lem:weak_star}
For all $0 \le r < s$,
\begin{equation}\label{eq:weak_star}
    C(s) - C(r) \;\ge\; 2\pi q^2\int_r^s \xi(t)\,dt
    \;+\; \frac{1}{16\pi}\int_r^s e^{t/2}\,D(t)\,dt .
\end{equation}
In particular $C(s) - C(r) \ge 2\pi q^2\int_r^s \xi(t)\,dt$, and if equality
holds here, then $D \equiv 0$ a.e.\ on
$(r,s)$.
\end{lemma}

\begin{proof}
Writing $F(t) := \int_{N_t}\big(H - \P(\nu,\nu)\big)^2
= 16\pi\big(1 - \varphi(t)\big)$ and substituting \eqref{eq:deficit} into
\eqref{eq:growth}, we obtain, for a.e.\ $r \ge 0$ and all $s > r$,
\[
    F(s) - F(r) \;\le\; -\frac12\int_r^s F(t)\,dt + 8\pi(s-r)
    - 32\pi^2 q^2\int_r^s\frac{dt}{|N_t|}
    - \int_r^s D(t)\,dt ,
\]
which in terms of $\varphi$ reads
\begin{equation}\label{eq:weak_phi}
    \varphi(s) - \varphi(r) \;\ge\; -\frac12\int_r^s\varphi(t)\,dt
    + 2\pi q^2\int_r^s\frac{dt}{|N_t|}
    + \frac{1}{16\pi}\int_r^s D(t)\,dt .
\end{equation}
In particular $\varphi$ is of bounded variation on compact intervals, and
its distributional derivative satisfies
$d\varphi \ge \big(-\tfrac12\varphi + 2\pi q^2|N_t|^{-1}
+ \tfrac{1}{16\pi}D\big)\,dt$ as measures. Multiplying by the integrating
factor $e^{t/2}$,
\[
    dC = e^{t/2}\,d\varphi + \tfrac12 e^{t/2}\varphi\,dt
    \;\ge\; e^{t/2}\left(\frac{2\pi q^2}{|N_t|}
    + \frac{D(t)}{16\pi}\right)dt
    = \left(2\pi q^2 \xi(t) + \frac{e^{t/2}D(t)}{16\pi}\right)dt ,
\]
and integrating over $(r,s]$ gives \eqref{eq:weak_star}. Since
$\E_0 = \E_0^+ = \E_0'$, the inequality extends to $r = 0$. The
final assertion follows since $D \ge 0$ and $e^{t/2} > 0$.
\end{proof}

\begin{lemma}\label{lem:g_monotone}
The function $\xi$ is nonincreasing.
\end{lemma}

\begin{proof}
By \cite[(48)]{Don26b} we have $e^{-r}|N_r| \le e^{-s}|N_s|$ for all
$0 \le r < s$, that is $|N_s| \ge e^{s-r}|N_r|$. Hence
\[
    \xi(s) = \frac{e^{s/2}}{|N_s|}
    \;\le\; \frac{e^{s/2}}{e^{s-r}|N_r|}
    = \frac{e^{r - s/2}}{|N_r|}
    \;\le\; \frac{e^{r/2}}{|N_r|} = \xi(r),
\]
the last inequality because $r \le s$. Equivalently, $\eta$ is
non-decreasing and $\xi = e^{-t/2}\eta^{-2}$.
\end{proof}

\begin{proposition}\label{prop:weak_function}
$W(t) \ge 0$ for all $t \ge 0$.
\end{proposition}

\begin{proof}
By Lemma \ref{lem:weak_star} and Lemma \ref{lem:g_monotone},
\[
    V(s) - V(r) = \big[C(s) - C(r)\big]
    + 4\pi q^2\big[\xi(r) - \xi(s)\big] \;\ge\; 0
\]
for all $0 \le r < s$, so $V$ is nondecreasing. Since $\Sigma$ is a past
apparent horizon, $H = \P(\nu,\nu)$ on $\Sigma$, whence
$\varphi(0) = 1$ and
\[
    V(0) = W(0) = 1 - \frac{4\pi q^2}{|\Sigma|} \;\ge\; 0
\]
by the hypothesis $|\Sigma| \ge 4\pi q^2$ of Theorem \ref{thm:main}.
Therefore $V(t) \ge 0$, and hence $W(t) \ge 0$, for all $t \ge 0$.
\end{proof}

Note that Lemma \ref{lem:g_monotone} and Lemma \ref{lem:weak_star} hold
across jump times, so no separate argument is required at a jump.

\begin{theorem}\label{thm:weak_monotonicity}
Let $(M^3,g,k;E,B)$ be a complete, connected, asymptotically flat
Einstein--Maxwell initial data set without boundary satisfying
\eqref{eq:charged_DEC}, \eqref{eq:E_B_div_free} and $\P \ge 0$. Let
$\E_0$ be a precompact open set whose boundary $\Sigma$ is a connected
$C^{2,\alpha}$-smooth outermost past apparent horizon with
$|\Sigma| \ge 4\pi q^2$. Then there is a weak solution
$(N_t)_{0 \le t < \infty}$ of \eqref{eq:flow} with initial condition
$\E_0$ such that, for all $0 \le r < s$,
\begin{equation}\label{eq:weak_mass_mono}
    m_{\P}^q(N_s) - m_{\P}^q(N_r)
    \;\ge\; \big(\eta(s) - \eta(r)\big)\,
    \frac{e^{r/2}\,W(r)}{4\sqrt{\pi}} \;\ge\; 0 .
\end{equation}
In particular $m_{\P}^q(N_t)$ is monotone nondecreasing.
\end{theorem}

\begin{proof}
Write $a := \eta(r)$ and $b := \eta(s)$, so that $b \ge a > 0$ by Lemma
\ref{lem:g_monotone}, and set $G := \int_r^s \xi(t)\,dt$. By
\eqref{eq:mass_eta_form} and $\eta\,\xi = |N_t|^{-1/2}$,
\[
    m_{\P}^q(N_s) - m_{\P}^q(N_r)
    = \frac{b\,C(s) - a\,C(r)}{4\sqrt{\pi}}
    + q^2\sqrt{\pi}\left(\frac{e^{-s/2}}{b} - \frac{e^{-r/2}}{a}\right).
\]
By Lemma \ref{lem:weak_star}, $C(s) \ge C(r) + 2\pi q^2 G$, and therefore
\[
    m_{\P}^q(N_s) - m_{\P}^q(N_r)
    \;\ge\; \frac{(b-a)\,C(r)}{4\sqrt{\pi}}
    + \frac{\sqrt{\pi}\,q^2\,b\,G}{2}
    + q^2\sqrt{\pi}\left(\frac{e^{-s/2}}{b} - \frac{e^{-r/2}}{a}\right).
\]
Since $\eta$ is nondecreasing, $\eta(t) \le b$ for $t \le s$, so that
$\xi(t) = e^{-t/2}\eta(t)^{-2} \ge e^{-t/2}b^{-2}$ and
\[
    G \;\ge\; \frac{1}{b^2}\int_r^s e^{-t/2}\,dt
    = \frac{2\big(e^{-r/2} - e^{-s/2}\big)}{b^2}.
\]
Substituting, the three charge terms collapse:
\[
    \frac{\sqrt{\pi}q^2 b G}{2}
    + q^2\sqrt{\pi}\left(\frac{e^{-s/2}}{b} - \frac{e^{-r/2}}{a}\right)
    \;\ge\; q^2\sqrt{\pi}\,e^{-r/2}\left(\frac{1}{b} - \frac{1}{a}\right)
    = -\,q^2\sqrt{\pi}\,e^{-r/2}\,\frac{b-a}{ab}.
\]
Hence, using $b \ge a$ once more in the form $(ab)^{-1} \le a^{-2}$,
\[
    m_{\P}^q(N_s) - m_{\P}^q(N_r)
    \;\ge\; (b-a)\left[\frac{C(r)}{4\sqrt{\pi}}
    - \frac{q^2\sqrt{\pi}\,e^{-r/2}}{ab}\right]
    \;\ge\; (b-a)\left[\frac{C(r)}{4\sqrt{\pi}}
    - \frac{q^2\sqrt{\pi}\,e^{-r/2}}{a^2}\right].
\]
Finally $a^2 = e^{-r}|N_r|$, so $e^{-r/2}a^{-2} = \xi(r)$ and the bracket
equals
\[
    \frac{1}{4\sqrt{\pi}}\Big(C(r) - 4\pi q^2 \xi(r)\Big)
    = \frac{V(r)}{4\sqrt{\pi}}
    = \frac{e^{r/2}W(r)}{4\sqrt{\pi}} ,
\]
which is nonnegative by Proposition \ref{prop:weak_function}. This proves
\eqref{eq:weak_mass_mono}.
\end{proof}

\begin{remark}
Letting $s \downarrow r$ along a smooth portion of the flow, where
$\eta'(t) = e^{-t/2} I(t)/\big(2|N_t|^{1/2}\big)$, the right-hand side of
\eqref{eq:weak_mass_mono} recovers the smooth bound
$\frac{d}{dt}m_{\P}^q(N_t) \ge
I(t)W(t)\big/\big(8\sqrt{\pi|N_t|}\big)$ of Section 3. In particular the
excess area growth $I$ is what the function $W$ has to pay for, and
\eqref{eq:weak_mass_mono} holds across jump times without further argument.
\end{remark}

\section{Proof of the Charged Penrose Inequality}\label{sec:proof}
\subsection*{Proof of main theorem.} 
Let $\E_0$ be the region enclosed by $\Sigma$, and choose the weak
solution of the $\P$-IMCF starting from $\E_0$ given by Theorem
\ref{thm:weak_monotonicity}. Since $\Sigma$ is outermost,
$\E_0 = \E_0'$ with respect to the weak solution field $\nu$, so the
Connectedness Lemma \cite[Lem.~2.8]{Don26a} and Theorem
\ref{thm:weak_monotonicity} apply. Because $\Sigma$ is a past apparent
horizon, $H = \P(\nu,\nu)$ on $\Sigma$, whence $\varphi(0) = 1$ and
\[
    m_{\P}^q(\Sigma)
    = \sqrt{\frac{|\Sigma|}{16\pi}} + q^2\sqrt{\frac{\pi}{|\Sigma|}} .
\]
Moreover Dong's mass
$m_{\P}(N_t) := \sqrt{|N_t|/16\pi}\,\varphi(t)
= \eta(t)C(t)/(4\sqrt{\pi})$ is a product of
nonnegative non-decreasing functions, by Lemmas \ref{lem:g_monotone},
\ref{lem:weak_star} and Proposition \ref{prop:weak_function}; together with
Lemma \ref{lem:g_monotone} this verifies the hypotheses of the asymptotic
comparison \cite[Lem.~7.2]{Don26b}, so that
$\lim_{t\to\infty} m_{\P}(N_t) \le m$. Finally
$|N_t| = e^{t}\eta(t)^2 \ge e^{t}|\Sigma| \to \infty$, so the charge term
$q^2\sqrt{\pi/|N_t|}$ vanishes in the limit and
$\lim_{t\to\infty} m_{\P}^q(N_t) = \lim_{t\to\infty} m_{\P}(N_t)$.
Combining,
\[
    \sqrt{\frac{|\Sigma|}{16\pi}} + q^2\sqrt{\frac{\pi}{|\Sigma|}}
    = m_{\P}^q(\Sigma)
    \le \lim_{t\to\infty} m_{\P}^q(N_t)
    = \lim_{t\to\infty} m_{\P}(N_t)
    \le m .
\]
We now consider the case of equality. Suppose first $q=0$. Then the equality
case of \cite[Thm.~1.2]{Don26a} gives $k\equiv 0$ and $(M,g)$ isometric to the
spatial Schwarzschild manifold; in particular $R\equiv 0$, so
$\mu = \tfrac{1}{16\pi}\big(R+(\operatorname{tr}k)^2-|k|^2\big) = 0$. Since
$\mu = \mu_m + \tfrac{1}{8\pi}\big(|E|^2+|B|^2\big)$ and $\mu_m \ge |J_m|_g \ge 0$
by \eqref{eq:charged_DEC}, both terms vanish; hence $E\equiv B\equiv 0$, and the
data is the canonical Schwarzschild slice. 

Assume from now on $q\neq 0$.
Note that by Lemma \ref{lem:weak_star}, $C(s) - C(r) > 0$ and by Lemma \ref{lem:g_monotone}, $\xi(r) - \xi(s) \geq 0$. This in turn yields
\[
V(s) - V(r)= \underbrace{[C(s) - C(r)]}_{>0} + \underbrace{4\pi q^2[\xi(r) - \xi(s)]}_{\geq 0} > 0
\]
Following the proof of Proposition \ref{prop:weak_function}, we get $W(t) > 0$ for all $t > 0$. Since $m^q_{\P}(N_s) = m^q_{\P}(N_r) = m$, we consider \eqref{eq:weak_mass_mono}:
\[
0 = m^q_{\P}(N_s) - m^q_{\P}(N_r) \geq \big(\eta(s) - \eta(r)\big)\,
   \underbrace{\frac{e^{r/2}\,W(r)}{4\sqrt{\pi}}}_{>0} \geq 0
\]
This implies $\eta(s) = \eta(r)$, hence $\eta(t) = const.$ and since $\eta(0) = |\Sigma|^{1/2}$, we obtain $|N_t| = e^t |\Sigma|$. Note that equality in the area-growth estimate II) forces $\P(\nu,\nu) = 0$ a.e. (the equality case of the $A(t)$-monotonicity, cf. \cite[proof of Thm. 1.2]{Don26a}).

Computing $C(s) - C(r)$ and the first term on the right hand side of \eqref{eq:weak_star} explicitly,
\begin{align*}
    &C(s) - C(r) = \frac{4\pi q^2}{|\Sigma|} (e^{-r/2} - e^{-s/2}) = 2\pi q^2 \int_r^s \xi(t)dt
\end{align*}
This yields
\[
 \frac{1}{16\pi}\int_r^s e^{t/2}\,D(t)\,dt  = 0
\]
so $D(t) = 0$ for a.e. $t$ and since $D(t)$ consists of non-negative terms, each vanishes, giving
\begin{enumerate}[(i)]
    \item $(H - \P(\nu,\nu))\tr_{N_t}\P = 0$ a.e. $t$
    \item $\int_{N_t} (\mu - J(\nu)) = \frac{2\pi q^2}{|N_t|}$
    \item $\chi(N_t) = 2$
\end{enumerate}
Since $H - \P(\nu,\nu) > 0$ a.e. on $N_t$ at regular times, $(i)$ implies $\tr_{N_t}\P =0$, and along with $\P(\nu,\nu) = 0$, this gives $k \equiv 0$. Applying $(ii)$ to the proof of Lemma \ref{lem:EM_pointwise}, 
\[
\mu - J(\nu) = \frac{1}{8\pi}(E_\nu^2 + B_\nu^2), \qquad \text{a.e. on }N_t
\]
and equality gives $J_m(\nu) = |J_m|_g$:
\[
    -\frac{1}{4\pi}\langle E^T \times B^T, \nu\rangle = \frac{1}{4\pi} |E\times B| \geq 0
\]
Moreover,
\[
\langle E^T \times B^T, \nu\rangle = |E^T||B^T| \geq 0
\]
Hence, $|E^T| = |B^T| = 0$ and as a result $\mu_m = |J_m|_g= 0$. By $\P(\nu,\nu) =0$, \eqref{eq:flow} is the classical inverse mean curvature
flow, and \eqref{eq:deficit} gives $\mathring{II} = 0$ and
$\nabla^N H = 0$, so that $H = H(t)$ is constant on each $N_t$ and
$II = \tfrac12 H\,g_{N_t}$. The argument of \cite[\S7.3]{Don26b} now applies
verbatim: each $N_t$ is smooth by elliptic regularity, and no jump can
occur, since the new portion of $N_t^+\setminus N_t$ would satisfy
$H = \P(\nu,\nu) = 0$ by \cite[(16)]{Don26a}, so that $N_t^+$ would
be a closed minimal surface, hence a generalized apparent horizon strictly
enclosing $\Sigma$ and disjoint from it by the strong maximum principle,
contradicting the assumption that $\Sigma$ is outermost. 
 
In contrast with \cite[\S7.3]{Don26b}, the scalar curvature does not vanish
here; instead,
\[
    R = 16\pi\mu = 2\big(|E|^2 + |B|^2\big) = 2\big(E_\nu^2 + B_\nu^2\big)
\]
is constant on each $N_t$. Since $H$ is constant on $N_t$, the evolution
equation \cite[Lem.~3.1]{Don26b} reduces to
$\partial_t H = -\big(\Ric(\nu,\nu) + |II|^2\big)/H$, whose
left-hand side depends only on $t$; as $|II|^2 = \tfrac12 H^2$ is likewise
constant on $N_t$, so is $\Ric(\nu,\nu)$. The Gauss equation
\[
    R^{N_t} = R - 2\Ric(\nu,\nu) + H^2 - |II|^2
    = R - 2\Ric(\nu,\nu) + \tfrac12 H^2
\]
therefore shows that $N_t$ has constant Gauss curvature, and since
$\chi(N_t) = 2$ by $(iii)$, each $N_t$ is a round sphere.

Since $II = \tfrac12 H\,g_{N_t}$, the induced metrics satisfy
$\partial_t\, g_{N_t} = \tfrac{2}{H}II = g_{N_t}$, so that
$g_{N_t} = e^{t}g_{\Sigma} = r^2 g_{S^2}$, where $r := r_+e^{t/2}$ and
$|\Sigma| = 4\pi r_+^2$. As the flow moves in the direction $\nu$ with speed
$1/H$ and $dr = \tfrac{r}{2}\,dt$, we obtain
\begin{equation}\label{eq:RN_form}
    g = \frac{1}{H^2}\,dt^2 + g_{N_t}
      = f(r)^{-1}dr^2 + r^2 g_{S^2},
    \qquad f(r) := \tfrac14 H^2 r^2 ,
\end{equation}
on $M\setminus\E_0$, with $f > 0$ for $r > r_+$ and $f(r_+) = 0$, since
$H$ vanishes on the minimal surface $\Sigma$.
 
The fields are normal with $E_\nu, B_\nu$ constant on each $N_t$,
so $\dv_g E = \dv_g B = 0$ gives $4\pi q_E = E_\nu\,|N_t| = 4\pi r^2E_\nu$
and likewise for $B$, whence
\[
    E_\nu = \frac{q_E}{r^2}, \qquad B_\nu = \frac{q_B}{r^2},
    \qquad R = 16\pi\mu = 2\big(|E|^2 + |B|^2\big) = \frac{2q^2}{r^4},
\]
using $\mu_m = 0$. On the other hand, for a metric of the form
\eqref{eq:RN_form} one computes
$R = \tfrac{2}{r^2}\big(1 - (rf)'\big)$, so that $(rf)' = 1 - q^2/r^2$ and
\[
    f(r) = 1 - \frac{2m_0}{r} + \frac{q^2}{r^2}
\]
for a constant $m_0$; expanding $g$ at infinity identifies $m_0$ with the
ADM mass $m$. Finally $f(r_+) = 0$ gives $r_+ = m + \sqrt{m^2-q^2}$, and
$m > |q|$, since otherwise $r_+ = m$ and
$\int_{r_+} f^{-1/2}\,dr = \int_{r_+} r\,(r-m)^{-1}dr$ diverges, so that
$\Sigma$ would lie at infinite distance and could not be a compact
boundary.
 
We conclude that $k \equiv 0$ and that $(M,g,E,B)$ is isometric to the
canonical slice $\{r \ge r_+\}$ of a sub-extremal Reissner--Nordstr\"om
spacetime with parameters $(m,q_E,q_B)$, the horizon being
$\Sigma = \{r = r_+\}$. Conversely, on such a slice the flow \eqref{eq:flow}
is by the round spheres $\{r = \mathrm{const}\}$ and
$m_{\P}^q(N_t) \equiv m$, so equality holds in \eqref{eq:main}. This
completes the proof of Theorem \ref{thm:main}. \qed
 
\begin{remark}
In particular equality forces $|\Sigma| > 4\pi q^2$: the boundary case
$|\Sigma| = 4\pi q^2$ of the area--charge hypothesis would give $r_+ = |q|$
and hence $m = |q|$, which was excluded above.
\end{remark}

\bibliographystyle{amsalpha}
\bibliography{references}

\end{document}